\documentclass[12pt, reqno]{amsart}

\usepackage{amssymb}
\usepackage{mathtools}
\usepackage{amsthm}
\usepackage{amsmath}
\usepackage{amsfonts}
\usepackage{amscd}
\usepackage[all,cmtip]{xy}

\theoremstyle{plain}
\newtheorem{thm}{Theorem}
      \newtheorem{lemma}[thm]{Lemma}
      \newtheorem{prop}[thm]{Proposition}
      \newtheorem{cor}[thm]{Corollary}

\def\AA{{\bf A}}

\theoremstyle{thm}

\begin{document}
\begin{abstract} 
We prove that every finite abelian group $G$ occurs as a subgroup of the class group of infinitely many real cyclotomic fields. 
\end{abstract}

\title[Class groups of real cyclotomic fields]{Class groups of real cyclotomic fields}

\author{Mohit Mishra}
\address{Harish-Chandra Research Institute, HBNI, Chhatnag Road, Jhunsi,  Allahabad 211 019, India.}\email{m.mishra0808@gmail.com}
\author{Rene Schoof}
\address{Università di Roma ``Tor Vergata'', Dipartimento di Matematica, Via della Ricerca Scientifica, I-00133 Roma, Italy}\email{schoof@mat.uniroma2.it}
\author{Lawrence C. Washington}
\address{Dept. of Mathematics, Univ. of Maryland, College Park, MD, 20742 USA}\email{lcw@umd.edu}
\maketitle

\section{\textbf{Introduction}} A natural question asks what groups occur as class groups of number fields.
Cornell \cite{GC3} proved that every finite abelian group occurs as a subgroup of the class group of some cyclotomic field. A natural related question is: does 
every finite abelian group occur as a subgroup for some real cyclotomic field? Class groups of real cyclotomic fields are quite small compared to the class groups 
of cyclotomic fields and it is interesting to find real cyclotomic fields with large class groups (see \cite{SC}). For a prime $p$, Cornell and Rosen (\cite{GC} and \cite{CR}) 
proved that every finite  elementary abelian $p$-group is embedded in the class group of infinitely many real cyclotomic fields. In \cite{Os2}, it is shown  that, for each $n\ge 1$, there are infinitely  
many real cyclotomic fields with class group containing a subgroup isomorphic to $(\mathbb Z/n\mathbb Z)^2$. To the best of our knowledge, there is no other result known in this direction. In this article, we extend these previous results and prove the following:

\begin{thm}\label{thm1}
Let $G$ be a finite abelian group. Then $G$ occurs as a subgroup of the class group of infinitely many real cyclotomic fields.  
\end{thm}

In the following, we first give a proof by lifting class groups from quadratic and cubic fields. Then, in the last section, we give  a second, completely different proof,
where the desired subgroup is constructed by a cohomological approach inspired by the theory of central class fields (see \cite{Fr}).

Our first construction
yields infinitely many fields with no containment relations among them. Our second proof shows that if we let $S$ be a finite set of primes, 
there is a real cyclotomic field $\mathbb Q(\zeta_n)^+$ with $G$ in its class group
such that $\mathbb Q(\zeta_n)^+/\mathbb Q$ is unramified at the primes in $S$.

\section{\bf Behavior of Class Groups under Field Extensions}
For a number field $F$, let $Cl_F$ denote the ideal class group of $F$.
\begin{lemma}\label{fundlemma}
Let $K$ and $L$ be Galois extensions of $\mathbb Q$ with $K\cap L=\mathbb Q$. Let $m=[K : \mathbb Q]$ and $n=[L : \mathbb Q]$.  Then $Cl_{KL}$ 
contains a subgroup isomorphic to
$Cl_{K}^{\, n}\oplus Cl_L^{\, m}$. 
\end{lemma}

\begin{proof} 

Consider the map
$$
\phi: Cl_{KL} \to Cl_K \oplus Cl_L, \mathbb \qquad  x \mapsto \left( N_{KL/K}(x), \, N_{KL/L}(x)\right).
$$
Let $I$ be an ideal of $K$. If we lift $I$ to $KL$ and then take its norm back to $K$, we obtain $I^n$. However, its norm from $KL$ to $L$ 
is the product of the conjugates of $I$ under $\text{Gal}(KL/L) = \text{Gal}(K/\mathbb Q)$,
hence is the lift of the norm of $I$ from $K$ to $\mathbb Q$. In particular, this norm is principal. It follows that the image of $\phi$ contains $Cl_K^{\, n} \oplus 0$.  
Similarly, it contains $0\oplus Cl_L^{\, m}$. Therefore,  it contains $Cl_K^{\, n}\oplus Cl_L^{\, m}$, so this group is a subquotient of $Cl_{KL}$. Therefore, it is also isomorphic to a subgroup
of $Cl_{KL}$.
\end{proof}

\begin{prop}\label{prop} Let $F$ be a number field with $\mathbb Q(\zeta_m)\subseteq F \subseteq \mathbb Q(\zeta_n)$ for some $m, n \ge 1$. Let $q$ be a prime
not dividing $[F : \mathbb Q(\zeta_m)]$. Then the norm map from the $q$-part of the ideal class group of $\mathbb Q(\zeta_n)$ to the $q$-part of the
ideal class group of $F$ is surjective.
\end{prop}
\begin{proof} 
We'll prove the following statement: Let $K$ be a Galois extension of $\mathbb Q(\zeta_m)$ for some $m\ge 1$ and let $p$ be a prime. Let $q$ be a prime
not dividing $[K: \mathbb Q(\zeta_m)]$. Then the norm map from the $q$-part of the ideal class group of $K(\zeta_{mp})$ to the $q$-part of the ideal class group of $K$
is surjective.

Since $[F(\zeta_{mr}) : \mathbb Q(\zeta_{mr})]$ divides $[F : \mathbb Q(\zeta_m)]$ for every $r\ge 1$, the statement implies the proposition by induction
by letting $K$ run through the tower of fields
$$
F \subseteq F(\zeta_{mp_1}) \subseteq F(\zeta_{mp_1p_2}) \subseteq \cdots \subseteq F(\zeta_n) = \mathbb Q(\zeta_n),
$$
where $p_1, p_2, \dots$ are the primes (with multiplicity) dividing $n/m$. 

We now prove the statement. The extension $\mathbb Q(\zeta_{mp})/\mathbb Q(\zeta_m)$ is totally ramified at $p$. Fix a prime of $K(\zeta_{mp})$ above $p$.
Its ramification index for $K(\zeta_{mp})/\mathbb Q(\zeta_m)$ equals its ramification index for $K(\zeta_{mp})/\mathbb Q(\zeta_{mp})$ times
its ramification index for $\mathbb Q(\zeta_{mp})/\mathbb Q(\zeta_m)$.
Since $q\nmid [K(\zeta_{mp}) : \mathbb Q(\zeta_{mp})]$, the $q$-part of the ramification index for $K(\zeta_{mp})/\mathbb Q(\zeta_m)$
equals the $q$-part of the ramification index for $\mathbb Q(\zeta_{mp})/\mathbb Q(\zeta_{m})$, which is the $q$-part of $[\mathbb Q(\zeta_{mp}) : \mathbb Q(\zeta_m)]$. It follows easily that the $q$-part of the inertia subgroup 
of $\text{Gal}(K(\zeta_{mp})/K)$ for this prime  is isomorphic to the $q$-part of its inertia subgroup in $\text{Gal}(\mathbb Q(\zeta_{mp})/\mathbb Q(\zeta_m))$.  Therefore 
the inertia subgroup has index prime to $q$ in $\text{Gal}(K(\zeta_{mp})/ K)$, which means that every subextension of $K(\zeta_{mp})/ K$ of $q$-power degree is totally ramified at the primes above $p$.

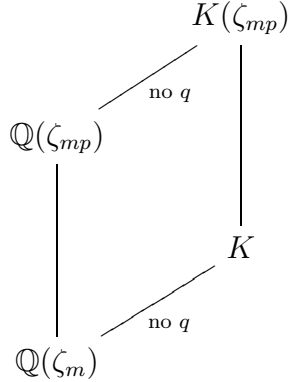
\begin{align*}
\xymatrix{
 & 
K(\zeta_{mp})\ar@{-}[dl]^{\text{no }q}\ar@{-}[dd]\\
 \mathbb Q(\zeta_{mp})\ar@{-}[dd]& \\
  & K\ar@{-}[dl]^{\text{no }q}\\
 {\mathbb Q(\zeta_m)}
}
\end{align*}

Let $H_q$ be the Hilbert $q$-class field of $K$; that is, $H_q$ is the maximal unramified abelian
$q$-extension of $K$. 
Then $H_q \cap K(\zeta_{mp}) / K$ is an unramified $q$-extension and therefore is trivial. 

Therefore, the Galois group of the $q$-class field of $K(\zeta_{mp})$ surjects by restriction onto $\text{Gal}(H_q/K)$. Restriction corresponds to the norm map on the 
$q$-parts of the ideal class groups
(see \cite[Appendix, Section 3, Theorem 5]{LCW}), so the norm map from the $q$-part of the ideal class group of $K(\zeta_{mp})$ to the $q$-part of the ideal class group
of $K$ is surjective. This proves the statement and the proposition. \end{proof}

\section{\bf Proof of Theorem 1}

Write $G= G_2 \oplus G_2'$, where $G_2$ is a 2-group, $G_2'$ has odd order.

It was proved by Yamamoto \cite{YY} and Weinberger \cite{We} that,  for any given $r$, there are infinitely many 
real quadratic fields with elements of order $r$ in their class groups. 
By Lemma \ref{fundlemma},
we can inductively construct a compositum $F'$ of real quadratic fields such that $Cl_{F'}$ contains a subgroup isomorphic to $G_2'$.
Let $n_1$ be the conductor of $F'$, so $F'\subseteq \mathbb Q(\zeta_{n_1})^+$. 
By Proposition \ref{prop},  the odd part of the class group of $\mathbb Q(\zeta_{n_1})$ surjects onto the odd part of the class group of $F'$ under the norm.

By a result of Uchida \cite{U},  for any given $r$, there are infinitely many cyclic cubic fields with elements of order $r$ in their class groups. 
By Lemma \ref{fundlemma},
we can inductively construct a compositum $F$ of cyclic cubic fields such that $Cl_F$ contains a subgroup isomorphic to $G_2$. We have $F\subseteq \mathbb Q(\zeta_{n_2})$
for some $n_2$.
Proposition \ref{prop} shows that the 2-part of the class group of $\mathbb Q(\zeta_{n_2})$ surjects via the norm onto the 2-part of the class group of $F$.

Let $n$ be the least common multiple of $n_1$ and $n_2$. By Proposition \ref{prop}, the norm maps the class group of $\mathbb Q(\zeta_n)$ onto the class group of $\mathbb Q(\zeta_{n_1})$. Therefore, the odd part of the class group
of $\mathbb Q(\zeta_n)$ surjects via the norm onto the odd part of the class group of $F'$. Since the norm factors through $\mathbb Q(\zeta_n)^+$,
the odd part of the class group of $\mathbb Q(\zeta_n)^+$ surjects onto the odd part of $Cl_{F'}$. Therefore, $Cl_{\mathbb Q(\zeta_n)^+}$ contains a subgroup
isomorphic to $G_2'$. 
The same argument shows that $Cl_{\mathbb Q(\zeta_n)^+}$ contains a subgroup
isomorphic to $G_2$, so the class group of $\mathbb Q(\zeta_n)^+$ contains a subgroup isomorphic to $G$. It is clear from the construction that we obtain infinitely many such fields. This completes the proof.

\section{\bf Remarks}

Here are a few examples related to Proposition \ref{prop} that  show  that the assumption $q\nmid [F : \mathbb Q(\zeta_m)]$ is needed. 

\begin{itemize}

\item[(i)] Consider $\mathbb{Q}(\zeta_{420})^+$. Then $\mathbb{Q}(\sqrt{105}) \subset \mathbb{Q}(\zeta_{420})^+$ and the class number of $\mathbb{Q}(\sqrt{105})$ is $2$. From \cite{JCM3}, we have the class number of $\mathbb{Q}(\zeta_{420})^+$ equal to $1$. (In \cite{IY}, it is claimed that for a square-free integer $m$, the 
class number of $\mathbb{Q}(\sqrt{m})$ divides the class number of $\mathbb{Q}(\zeta_{4m})^+$. This result is not true. This example and the next serve as counterexamples. In \cite{Os}, it is shown how to correct the statement.)

\item[(ii)] Again from \cite{JCM3}, the class numbers of $\mathbb{Q}(\zeta_{220})^+$ and $\mathbb{Q}(\sqrt{55})$ are equal to $1$ and $2$, respectively. Also $\mathbb{Q}(\sqrt{55}) \subset \mathbb{Q}(\zeta_{220})^+$. 

\item[(iii)] The cyclic cubic  field defined by the polynomial $x^3 - 44x^2 + 524x - 944$ has class number 3 and is contained in $\mathbb Q(\zeta_{91})^+$, which has class number 1
(see \cite{vdL}). This shows that the 3-part of the class group of a cubic field can disappear when lifted to a cyclotomic field.

\end{itemize} 

 \section{\bf Strengthening Proposition 3}

By using some more machinery, we can sharpen the statement of Proposition 3. 

\begin{thm}\label{normthm} Let $F$ be a number field with $\mathbb Q(\zeta_m)\subseteq F \subseteq \mathbb Q(\zeta_n)$ for some $m, n \ge 1$ and let 
$d=[F : \mathbb Q(\zeta_m)]$. Let $N$ be the norm map from $Cl_{\mathbb Q(\zeta_n)}$ to $Cl_F$. Then $C_F/\text{Image}(N)$ has exponent diving $d$.
\end{thm}
\begin{proof} 

We start with a result that is part of what is known as ``genus theory.''

\begin{prop} Let $F$ be as in Theorem \ref{normthm} and let $K$ be the maximal extension of $F$ that is unramified at all finite primes and 
is abelian over $\mathbb Q$. Then $\text{Gal}(K/\mathbb Q(\zeta_m))$ has exponent dividing $d$.
\end{prop}

\begin{proof} A Dirichlet character $\chi$ can be written uniquely in the form $\chi = \prod \chi_p$, where $p$ runs through the primes and
the character $\chi_p$ has  conductor a power of $p$. 
Let $X$ be the group of Dirichlet characters associated to $F$. For each prime $p$, let $X_p$ be the group of Dirichlet characters generated by the characters $\chi_p$
as $\chi$ runs through the characters in $X$. Leopoldt's character theory (see \cite[Chapter 3]{LCW}) says that the ramification index for $p$ in $F/\mathbb Q$ 
is $e=\#X_p$. Let $\widetilde{X}$ be the group of Dirichlet characters generated by the set of $X_p$ as $p$ runs through the primes, and let
$L$ be the number field corresponding to $\widetilde{X}$. Since $X_p$ is the set of $p$-components of the characters in $\widetilde{X}$, we have
$\widetilde{X}_p=X_p$, so $L/F$ is unramified at all finite primes.

Clearly $L\subseteq K$.   Let $Y$ be the set of Dirichlet characters associated to $K$.  Since $K/L$ is unramified, we must have $Y_p =\widetilde{X}_p$ for each $p$.
Since $\widetilde{X}$ contains each $X_p$, we must have $\widetilde{X}\supseteq Y$, which implies that $K\subseteq L$, hence $K=L$. 

Since $\mathbb Q(\zeta_m)$ is the compositum of the prime power cyclotomic fields that it contains, the characters of $\mathbb Q(\zeta_m)$ are products of the characters of prime power
conductor associated to $\mathbb Q(\zeta_m)$.  If $\chi\in X$, then $\chi^d$ is a character associated to $\mathbb Q(\zeta_m)$. 
Write $\chi=\prod \chi_p$, so $\chi^d = \prod \chi_p^d$. The $p$-component is $\chi_p^d$, which is therefore a character of $\mathbb Q(\zeta_m)$. 
It follows that the $d$th power of each character in $\widetilde{X}$ is a character of $\mathbb Q(\zeta_m)$. Therefore, 
$\text{Gal}(K/\mathbb Q(\zeta_m))$ is annihilated by
the $d$th power map. \end{proof}

\begin{cor}  Let $K'$ be the maximal extension of $F$ that is unramified at all primes and is abelian over $\mathbb Q$. Then $\text{Gal}(K'/\mathbb Q(\zeta_m))$ has exponent dividing $d$.
\end{cor}

\begin{proof} Since $K'\subseteq K$, the corollary follows from the proposition. \end{proof}

We can now finish the proof of Theorem \ref{normthm}.  Let $H_F$ be the Hilbert class field of $F$ and $H_n$ be the Hilbert class field of $\mathbb Q(\zeta_n)$. Then
$$
\text{Gal}(H_n/\mathbb Q(\zeta_n)) \twoheadrightarrow \text{Gal}(H_F(\zeta_n)/\mathbb Q(\zeta_n)) \simeq \text{Gal}(H_F/(\mathbb Q(\zeta_n)\cap H_F).
$$
Restriction corresponds to the norm on ideal class groups. The corollary implies that 
$$\text{Gal}(H_F/F)^d\subseteq \text{Gal}(H_F/K')\subseteq \text{Gal}(H_F/(\mathbb Q(\zeta_n)\cap H_F)),$$
so we find that $C_F^d$ is in the image of $N$. \end{proof}

\section{\bf A cohomological approach}

In this section we prove Theorem \ref{thm1} using a completely different approach.

Let $k\subset K$ be a finite Galois extension of number fields with $G={\rm Gal}(K/k)$.
The Tate cohomology group  $\widehat
H^{-1}(G,Cl_K)$ of the class group $Cl_K$ is isomorphic to $Cl_K[N]/I\,Cl_K$ where
$Cl_K[N]$ denotes the kernel of the $G$-norm map $N:Cl_K\rightarrow Cl_K$ and
$I\,Cl_K$ denotes the group generated by the
$\sigma(x)/x$ for $\sigma\in G$ and $x\in Cl_K$.
In particular,  $\widehat H^{-1}(G,Cl_K)$ is  a subquotient of $Cl_K$. We show that it may be large.
\medskip

Let $O_K^*$ denote the unit group of the ring of integers of $K$.
Let $I_K$ denote the group of fractional ideals and $P_K$ its subgroup of principal fractional ideals.
Let $\AA_K^*$ denote the idele class group, $U_K$ its subgroup of unit ideles, and $C_K$ the group of idele classes.
Put $Q_K=U_K/O_K^*$. Then we have the following commutative diagram with exact rows and columns \cite[IX, Section 3]{CF}:
\begin{equation}\label{mainCD}
\begin{CD} 
@. 1 @. 1 @. 1 @. \\
@. @VVV @VVV @VVV @.\\
1 @>>> {\mathcal O}_K^{\times} @>>> K^{\times} @>>> P_K @>>> 1\\
@. @VVV @VVV @VVV @.\\
1 @>>> U_K @>>> \AA_K^*@>>> I_K @>>>1\\
@. @VVV @VVV @VVV @. \\
1 @>>> Q_K @>>> C_K @>>> Cl_K @>>> 1\\
@. @VVV @VVV @VVV @.\\
@. 1 @. 1 @. 1 @.
\end{CD}
\end{equation}
The Galois group $G$ acts on the groups in this diagram.
The long exact  Tate cohomology  sequences of the rows and columns give rise to an infinite commutative diagram,
a piece of which is the following commutative diagram with exact rows:
\smallskip
\begin{equation}\label{2ndCD}
\begin{CD}
\widehat{H}^{-1}(G,U_K) @>>>  \widehat{H}^{-1}(G,\AA^*_K) @>>> 0\\
@VV\phi V@VV\psi V @VVV \\
\widehat{H}^{-1}(G,Q_K) @>>> \widehat{H}^{-1}(G,C_K) @>>> \widehat{H}^{-1}(G,Cl_K)
\end{CD}
\end{equation}
\smallskip\noindent
Note that the group $\widehat{H}^{-1}(G,I_K)$ is zero. This follows from Shapiro's lemma and the fact that the cohomology groups $\widehat{H}^{-1}(G_w,\mathbb Z)$ vanish.
See \cite[IV,  Prop. 2]{CF} and \cite[IX, Section 3]{CF}. Here $w$ denotes a prime of $K$ lying over a prime $v$ of $k$ and $G_w$ denotes the local Galois group ${\rm Gal}(K_w/k_v)$.

Let $r_v$  be the cardinality of a minimal set of generators of the `Schur multiplier' $H_2(G_w,\mathbb Z)=\widehat{H}^{-3}(G_w,\mathbb Z)$, which is independent of the choice of $w$ above $v$.

\begin{prop}   \label{propSchur}
The group  $\widehat{H}^{-1}(G,Cl_K)$ admits a subquotient that is isomorphic to the homology group
$H_2(G,\mathbb Z)$ modulo a subgroup generated by at most $n+\sum_v r_v$ elements. Here  $n=[k:\mathbb Q]$ and $v$ runs over the  finite ramified primes of $k$.
\end{prop}

\begin{proof}  An easy diagram chase shows that the induced sequence of the cokernels
of the vertical arrows in the diagram (\ref{2ndCD})  is exact:
$$
{\rm cok}\,\phi\longrightarrow{\rm cok}\,\psi\longrightarrow\widehat{H}^{-1}(G,Cl_K).
$$
It follows that the group $\widehat{H}^{-1}(G,Cl_K)$ admits a subquotient isomorphic to the quotient of ${\rm cok}\,\psi$ 
by the image of ${\rm cok}\,\phi$. The group 
${\rm cok}\,\psi$ is isomorphic to $\widehat{H}^{-1}(G,C_K)/\psi\widehat{H}^{-1}(G,\AA^*_K)$.
By class field theory, $\widehat{H}^{-1}(G,C_K)$ is isomorphic to $\widehat{H}^{-3}(G,\mathbb Z)=H_2(G,\mathbb Z)$ (see \cite[VII, Section 11.3]{CF}).
Similarly, by Shapiro's lemma $\widehat{H}^{-1}(G,\AA^*_K)$ is the sum over the primes $v$ of $k$ 
of the Schur multipliers $H_2(G_w,\mathbb Z)$ of the local Galois groups $G_w$.
Note that for infinite primes and for unramified primes $v$, the group $G_w$ is cyclic and hence $H_2(G_w,\mathbb Z)$ is zero.
Therefore $\widehat{H}^{-1}(G,\AA^*_K)$ is a  sum over the ramified primes $v$ of $k$ of the groups~$H_2(G_w,\mathbb Z)$.
It follows that $\psi \widehat{H}^{-1}(G,\AA^*_K)$ can be generated by $\sum_vr_v$ elements.

By the exactness of the leftmost column of the diagram (\ref{mainCD}), the group ${\rm cok}\,\phi$ is a subgroup 
of $\widehat{H}^0(G,O_K^*)$, which in turn is a quotient of the unit group $O_k^*$.
By Dirichlet's unit theorem, the group $O_k^*$  can be generated by at most $n=[k:\mathbb Q]$ elements.
Therefore the same is true for~${\rm cok}\,\phi$. This proves the proposition. \end{proof}

\medskip
Now we specialize to $k=\mathbb Q$ and $K$ a finite abelian extension of $\mathbb Q$. 
Since the Galois group $G={\rm Gal}(K/k)$ is abelian, its Schur multiplier $H_2(G, \mathbb Z)$ is isomorphic to the exterior 
product~$G\wedge G$. See \cite[Thm.6.4]{Brown} and \cite[Thm. 3]{CM}.

\begin{cor}  Let $t\ge 4$ and $M\ge 2$ and let $f$ be a product of $t$ distinct primes that are congruent to $1$ modulo $M$.
Then the class group of the cyclotomic field $\mathbb Q(\zeta_f)$ admits a subgroup isomorphic to $(\mathbb Z/M\mathbb Z)^T$, where $T=\frac{t^2-3t-2}{2}$.
\end{cor}

\begin{proof} The Galois group $G$ of $K=\mathbb Q(\zeta_f)$ over $\mathbb Q$ is isomorphic to $(\mathbb Z/f\mathbb Z)^*$.
 Since $f$ is the product of $t$ primes that are congruent to $1$ modulo $M$, the group $G$ admits a  subgroup isomorphic to 
 $(\mathbb Z/M\mathbb Z)^{t}$. Therefore its Schur multiplier admits a subgroup isomorphic to $(\mathbb Z/M\mathbb Z)^{t\choose 2}$.
 
 There are precisely $t$ finite primes $v$ of $\mathbb Q$ ramified in $\mathbb Q(\zeta_f)$.
For a prime $w$ of $K$ lying over a finite ramified prime  $v$ of $\mathbb Q$,  local class field theory implies that the Galois group $G_w$  
is a  quotient  of $\mathbb Q_v^*$. Therefore  $G_w$ can be generated by three or two elements, depending on whether $v=2$ or not.
It follows that the Schur multiplier of $G_w$ is cyclic and $r_v=1$ when $v\not=2$.  
For  $v=2$ it is the product of a cyclic group and a group of order at most~$4$, so that $r_v\le 3$. Since $M\ge 2$, we have $v\not=2$ in our situation.
                                                
 By Proposition \ref{propSchur}, the group $\widehat{H}^{-1}(G,Cl_K)$ admits therefore a subquotient isomorphic to $(\mathbb Z/M\mathbb Z)^{t\choose 2}$ modulo
 a subgroup generated by at most $t+1$ elements. So it admits a subquotient isomorphic to a product of
$T={t\choose 2}-(t+1)$ copies of $\mathbb Z/M\mathbb Z$. Since $\widehat{H}^{-1}(G,Cl_K)$ is a subquotient of $Cl_K$,  the same is  true for the class group. But then
 $Cl_K$ also has a subgroup isomorphic to $(\mathbb Z/M\mathbb Z)^T$, as required. \end{proof}
 
 \begin{cor} Let $t\ge 4$ and $M\ge 2$ and let $f$ be the product of $t$ primes that are congruent to $1$ modulo $2M$.
Then the class group of $\mathbb Q(\zeta_f+\zeta_f^{-1})$ admits a subgroup isomorphic to $(\mathbb Z/M\mathbb Z)^T$, where $T=\frac{t^2-3t-2}{2}$.
\end{cor}

\begin{proof}  The Galois group $G$ of $K=\mathbb Q(\zeta_f+\zeta_f^{-1})$ over $\mathbb Q$ is isomorphic to $(\mathbb Z/f\mathbb Z)^*/\{\pm 1\}$.
 Since $f$ is the product of $t$ primes that are congruent to $1$ modulo $2M$, the group $G$ admits a quotient and hence a subgroup isomorphic to 
 $(\mathbb Z/M\mathbb Z)^{t}$. The rest of the proof is the same as the one of the previous corollary.
 \end{proof}

\section*{\textbf{Acknowledgments}} 
\noindent The first author's work was partially supported by Infosys grant. He would like to thank his advisor Prof. K. Chakraborty for his constant support and guidance. 
%%%%%%%%%%%%%%%%%%%%%%%%%%%%%%%%%%%%%%%%%%%%%%%%%%%      

\end{document}